\documentclass[12pt]{article}
\usepackage{amsthm}
\usepackage{picinpar}
\usepackage{tikz}
\usepackage{graphicx}
\usepackage{amssymb,amsmath, mathrsfs}
\newtheorem{teorema}{Theorem}

\newtheorem{proposizione}[teorema]{Proposition}

\newtheorem{definizione}[teorema]{Definition}
\theoremstyle{remark}

\title{A dynamical system approach to the Kakutani-Fibonacci sequence}

\author{ {\it Ingrid Carbone,  Maria Rita Iac\`{o}, Aljo\v{s}a Vol\v{c}i\v{c}}}
\date{}

\begin{document}

\maketitle

\begin{abstract}
 In this paper we consider the sequence of Kakutani's $\alpha$-refinements  corresponding to the inverse of golden ratio (which we call {\it Kakutani-Fibonacci sequence of partitions}) and associate to it an ergodic interval exchange  (which we call {\it Kakutani-Fibonacci transformation}) using the ``cutting-stacking" technique. We prove that the orbit of the origin under this map coincides with a low discrepancy sequence (which we call {\it Kakutani-Fibonacci sequence of points}), which has been also considered by other  authors. 

\end{abstract}


\section{Introduction and preliminaries}

In this paper we study with a dynamical system approach a sequence of points which arises from a beautiful geometric construction due to Kakutani \cite{Kakutani}: the $\alpha$-splitting procedure. 
He proved the uniform distribution of successive $\alpha$-refinements using ergodic theory methods.
It has been proved in \cite{Carbone} that when $\alpha$ is the inverse of the golden ratio, the sequence of $\alpha$-refinements (called {\it Kakutani-Fibonacci sequence of partitions} for reasons which will be clear later) has {\it low discrepancy}.

In \cite{Carbone} and \cite{Carbone2} the first author proposed two algorithms which associate to this sequence of $\alpha$-refinements a sequence of points (called {\it Kakutani-Fibonacci sequence of points}) which has {\it low discrepancy}, too.

In this section we present some basic definitions on uniform distribution, discrepancy,  ergodic theory and the well-known Birkhoff's Theorem. We also recall the definition of the {\it Kakutani-Fibonacci sequence}, with the addition of some properties, the most important of which being Theorem \ref{5}.  

In Section 2 we introduce an interval exchange $T$ associated to the cutting-stacking procedure for this sequence and we prove that this transformation is ergodic (Theorem \ref{ergodicity}). By means of this result and Bikhoff's Theorem, we conclude that the orbit of almost every $x\in [0,1[$ is a uniformly distributed sequence of points (Theorem \ref{b-k-f}). 
With the last results (Theorem \ref{15}) we obtain the Kakutani-Fibonacci sequence as the orbit of the origin under $T$.
In a way, our results parallel Lambert's observation concerning the von Neumann-Kakutani mapping and the van der Corput sequence \cite{Lambert}. 

Both sections are endowed by figures which should be of some help for the reader.  

\begin{definizione}\label{1}
\rm{Let $\alpha \in\ ]0,1[$ be a real number and $\pi$ be any partition of $[0,1[\ $. Kakutani's  $\alpha$-refinement of $\pi$, denoted by $\alpha\pi$, is obtained by splitting all the intervals of $\pi$ having maximal length into subintervals of lengths proportional to $\alpha$ and $1-\alpha$, respectively.\\
The sequence of partitions $\{\alpha^{n}\omega\}_{n \ge 1}$ obtained by successive $\alpha$-refinements of the trivial partition $\omega = \{[0, 1[\}$ is called the {\it Kakutani $\alpha$-sequence}.}
\end{definizione}

\rm{\begin{definizione}  \label{2}
 {\rm   Let $\{\pi_n\}_{n \ge 1}$ be a sequence of partitions of $[0,1[$, represented by $\pi_n=\big([y_{j}^{(n)},y_{j+1}^{(n)}[\ :1\leq j\leq t_n\big)$. 
 The {\it discrepancy} of  $\{\pi_n\}_{n \ge 1}$ is defined by
\begin{equation*} D(\pi_n) = \sup_{0 \le a < b \le 1}  \left | \frac {1} {{t_n}} \sum _{j=1}^{{t_n}} \chi _{[a, \,b[} (y_j^{(n)}) - (b-a)\right | .\end{equation*}

\noindent We say that $\{\pi_n\}_{n\ge 1}$ is \emph{uniformly distributed (u.d.)} if $D(\pi_n)\rightarrow 0$ when $n \rightarrow \infty$.

\noindent If there exists a constant $C>0$ such that $ {t_n}\,D(\pi_n)\le C\,$ for any $n$, we say that $\{\pi_n\}_{n \ge 1}$ has {\it low discrepancy}.
}
\end{definizione}

\begin{definizione} \label{3} 
 {\rm  The {\it discrepancy} of  the first $N$ points of a sequence  $X=\{x_n\}_{n \ge 1}$ in $[0, 1[$ is defined by
\begin{equation*} D_N(X) = \sup_{0 \le a < b \le1}  \left | \frac {1} {N} \sum _{j=1}^{N} \chi _{[a,\, b[} (x_j) - (b-a)\right |. \end{equation*}

\noindent We say that $X$ is \emph{uniformly distributed (u.d.)} if $D_N(X)\rightarrow 0$ when $n \rightarrow \infty$.

\noindent If there exists a constant $C>0$ such that $ ND_N(X)\le C\, \log N$ for any $N$, we say that $X$ has {\it low discrepancy}.
}
\end{definizione}

The following theorem shows the important role of u.d. sequences of points in Quasi-Monte Carlo methods.

\begin{teorema}\label{4}
\rm{A sequence of points $\{x_n\}_{n \ge 1}$ in $[0,1[$ is   \emph{uniformly distributed (u.d.)} if and only if for any continuous function $f$ on $[0,1]$ it is 
$$\lim_{N\to\infty}\frac{1}{N}\sum_{j=1}^{N}f(x_j)=\int_0^{1}f(t) dt \ .$$}
\end{teorema}

 Kakutani  proved that the sequence of partitions $\{\alpha^{n}\omega\}_{n \ge 1}$ is uniformly distributed \cite{Kakutani}. 
 
 This result got a considerable attention in the subsequent decades and in the recent years  many extensions of Kakutani's idea have been considered.

The concept of u.d. sequences of partitions has been extended to compact metric probability spaces \cite{Chersi-Volcic} and to fractals \cite{I-V}. Recently, the first and third authors showed how to rearrange a dense sequence of partitions in order to obtain a u.d. one \cite{Carbone_Volcic2}.

The third author \cite{Volcic} extended Kakutani's $\alpha$-refinement by introducing the concept of {\it $\rho$-refinement} of a partition $\pi$ of $[0,1[$. It is obtained by splitting the longest intervals of $\pi$ into a finite number of parts homothetically to a fixed finite partition $\rho$ of $[0,1[$. It has been proved that the sequence of subsequent $\rho$-refinements $\{\rho^{n}\omega\}_{n \ge 1}$ is u.d. for every $\rho$.

Kakutani's splitting procedure has also been extended to the multidimensional case in \cite{Carbone_Volcic}.

The general case in which the successive $\rho$-refinements are applied to an arbitrary non trivial partition $\pi$ has been studied in \cite{Aistleitner_Hofer}. The authors give necessary and sufficient conditions on $\pi$ and $\rho$ which assure that  $\{\rho^{n}\pi\}_{n \ge 1}$ is u.d..

Drmota and Infusino \cite{Drmota_Infusino} derived bounds for the discrepancy of the sequences of $\rho$-refinements  $\{\rho^{n}\omega\}_{n \ge 1}$ considering a new approach based on the Khodak algorithm for parsing trees.

 A specific countable class of sequences of $\rho$-refinements,  called $LS$-sequen\-ces of partitions, has been studied in \cite{Carbone}.
 In that paper it was also presented an explicit algorithm which orders the points determining the $LS$-sequences of partitions to obtain a u.d. sequence of points, called $LS$-sequence of points, providing estimates for the discrepancy. 
 
 The paper \cite{Carbone2} contains  a new explicit algorithm \lq\lq \`{a} la van der Corput\rq\rq\ to generate these sequences of points using the representation of natural numbers in base $L+S$.  This algorithm is explicit and allows us to compute directly the points of each $LS$-sequence.
 
 Some interesting extensions of $LS$-sequences of points to dimension two have been first introduced in \cite{tesi} and presented in \cite{C-I-V}.
 
 \bigskip
The sequence of points we want to study in this paper is associated to Kakutani's $\alpha$-sequence $\{\alpha^n \omega\}_{n\ge 1}$, where $\alpha=\frac{\sqrt{5}-1}{2}$ is the inverse of the golden ratio. 
It is easy to see that, since $\alpha^2+\alpha=1$, each partition $\alpha^n \omega$ is formed only by intervals of length $\alpha^n$ (long) and $\alpha^{n+1}$ (short).

If we denote by $\{ l_n\}_{n \ge 1}$ and $\{ s_n\}_{n \ge 1}$ the number of long and short intervals, respectively, and by $\{ t_n \}_{n \ge 1}$ the total number of intervals of the $n$-th partition, they satisfy the difference equation 
\begin{equation*}a_{n+2}=a_{n+1}+a_n
 \end{equation*}
 with initial conditions $l_0=1$ and $l_1=1$, $s_0=0$ and $s_1=1$ and, respectively, $t_0=1$ and $t_1=2$. 
 Therefore $\{ t_n \}_{n \ge 1}$ is the Fibonacci sequence. 
 For this reason, in \cite{Carbone} this sequence has been called the {\it Kakutani-Fibonacci sequence of partitions}, and it was also proved the following

 \begin{teorema} The Kakutani-Fibonacci sequence of partitions $ \{\alpha^n \omega\}_{n\ge 1}$ has low discrepancy.
 
 \end{teorema}
  
 In \cite{Carbone} and \cite{Carbone2} the first author proposed two algorithms which order in a natural way the left endpoints of  $ \{\alpha^n \omega\}_{n\ge 1}$ and associate to it a sequence of points called the {\it Kakutani-Fibonacci sequence of points} we will denote by $\{\xi_n\}_{n \ge 1}$. 
 
  The sequence $\{\xi_n\}_{n \ge 1}$  can be  constructed inductively as follows. 
  
  Begin with the first block
  \begin{equation*} \Lambda_1=( \xi_1, \xi_2)= (0, \alpha ).
 \end{equation*}
 
 \noindent If we denote by  $\Lambda_n=(\xi_1,\dots, \xi_{t_n})$ the set of the $t_n$ points of $\alpha^n\omega$ ordered in the appropriate way, the $t_{n+1}=t_n+l_n$ points of $\alpha^{n+1}\omega$ are defined by
\begin{equation}\label{1}\Lambda_{n+1}=
\big(\xi_1^,\dots,\xi_{t_n}, \phi_{n+1}(\xi_1),\dots ,\phi_{n+1}(\xi_{l_n})\big)\ ,\end{equation}
where $\phi_{n+1}(x)=x+\alpha^{n+1}$.

\bigskip
In Figure 1 we  see  the first 8 elements of the sequence $\{\xi_n\}_{n \ge 1}$.

\begin{figure}[h!]

\begin{center}
\begin{tikzpicture}[scale=13]
\draw (0, 0) -- (1,0);
\draw (0,-0.01) node[below, black]{0} -- (0,0.01) node[above, black]{$\xi_1$};
\draw (1,-0.01) node[below, black]{1} -- (1,0.01);
\draw (0.61803,-0.01) node[below, black]{\begin{scriptsize} $\alpha$\end{scriptsize}}-- (0.61803,0.01) node[above, black]{$\xi_2$};
\draw (0.38196,-0.01) node[below, black]{\begin{scriptsize}$\alpha^{2}$\end{scriptsize}}-- (0.38196,0.01) node[above, black]{$\xi_3$};
\draw (0.23606,-0.01) node[below, black]{\begin{scriptsize}$\alpha^{3}$\end{scriptsize}}-- (0.23606,0.01) node[above, black]{$\xi_4$};
\draw (0.85410,-0.01) node[below, black]{\begin{scriptsize}$\alpha +\alpha^{3}$\end{scriptsize}}-- (0.85410,0.01) node[above, black]{$\xi_5$};
\draw (0.14589,-0.01) node[below, black]{\begin{scriptsize}$\alpha^{4}$\end{scriptsize}}-- (0.14589,0.01) node[above, black]{$\xi_6$};
\draw (0.76393,-0.01) node[below, black]{\begin{scriptsize}$\alpha +\alpha^{4}$\end{scriptsize}}-- (0.76393,0.01) node[above, black]{$\xi_7$};
\draw (0.52786,-0.01) node[below, black]{\begin{scriptsize}$\alpha^2 +\alpha^{4}$\end{scriptsize}}-- (0.52786,0.01) node[above, black]{$\xi_8$};
\end{tikzpicture}
\end{center}
\caption{First 8 points of the Kakutani-Fibonacci sequence $\{\xi_n\}_{n \ge 1}$}
\end{figure}
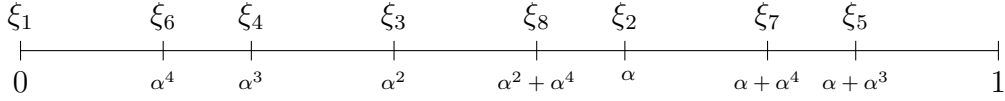

It is interesting to notice that the sequence $\{\xi_n\} _{n \ge 1} $ has been constructed by Ninomiya in \cite{Ninomiya} using a completely different approach, without observing that it is one of the possible reordering of the points of a Kakutani $\alpha$-sequence. 

The next result has been proved independently by Ninomiya  in \cite{Ninomiya} and by the first author in \cite{Carbone} using completely different methods.

\begin{teorema}\label{5}
The Kakutani-Fibonacci sequence of points $\{\xi_n\}_{n \ge 1}$ has low discrepancy.
\end{teorema}

 In the next section we will obtain this sequence using a dynamical system approach.

	This approach to uniform distribution theory goes back to von Neumann and Kakutani (see \cite{Friedman}, \cite{Friedman2}, \cite{Lambert}).
 For a recent survey on this matter see \cite{grahellia}.

\begin{definizione}\label{6}
\rm{Let $X$ be a  set and $T:X\rightarrow X$ be a map. Given $x\in X$, the sequence $x, T(x), T^{2}(x),\ldots$ is called the \emph{orbit} of $x$ in $X$.}
\end{definizione}

\begin{definizione}\label{7}
\rm{
A map $T:[0,1[ \ \rightarrow [0,1[$ is said to be an \emph{interval exchange} on $[0,1[$ if there exist a finite or countable family $\mathcal{F}$ of non empty subintervals $I_k=[a_k,b_k[$ of $[0,1[$ and a family of corresponding real numbers $c_ k$ with $1\leq k \leq card(\mathcal{F})$, such that:
\begin{itemize}
\item [(i)] $I_h \cap I_k = \emptyset$ and $(I_h+c_ h)\cap (I_h+c_ k)=\emptyset$ if $h\neq k$,
\item[(ii)] $T(x)=x+c_ k$ if $x\in I_k$,
\item[(iii)] $\lambda(\cup_{k\geq 0}I_k)=\lambda(\cup_{k\geq 0}(I_k+c_ k))=1$.
\end{itemize}
}
\end{definizione}

\bigskip
It is easy to see that an interval exchange is a map of $[0,1[$ into itself which is one-to-one, preserves the Lebesgue measure $\lambda$ and is continuous $\lambda$-almost everywhere.

  A special role is played by the \emph{ergodic} transformations (see \cite{Friedman}, for instance).

\begin{definizione}
\rm{Let $(X,\mathcal{A},\mu)$ be a probability space. A  measurable transformation  $T:X\rightarrow X$   is said \emph{ergodic} if  for every $A\in\mathcal{A}$ such that 
$T^{-1}(A) = A$, either $\mu(A)=0$ or $ \mu(A)=1$.} 
\end{definizione}

We will use later an equivalent property: $T$ is ergodic if and only if for any measurable set $B$ such that $\mu(B)>0$, the set $B^T=\bigcup_{i=- \infty}^{+ \infty}T^i(B)$ has measure $1$.

We will also assume that the ergodic transformation $T$ is measure preserving, as most (but not all) authors do.

The system $(X,\mathcal{A},\mu, T)$  
  is called a \emph{measure theoretical dynamical system}, or \emph{dynamical system}, for short.
 If $T$ is ergodic, the system is called {\it ergodic}.

\bigskip
The link between dynamical systems and uniform distribution is given by the following  

\begin{teorema}[Birkhoff's Theorem]\label{B}
Let $(X,\mathcal{A},\mu, T)$ be a measure theoretical dynamical system. Then, for every $f \in \mathcal{L}^1(X)$ $$\lim_{N\to\infty}\frac{1}{N}\sum_{j=0}^{N-1}f(T^{j}x)$$ exists for $\mu$-almost every $x\in X$ (here $T^0x=x$).

\noindent If $T:X\rightarrow X$ is ergodic, then for every $f\in \mathcal{L}^{1}(X)$ we have $$\lim_{N\to\infty}\frac{1}{N}\sum_{j=0}^{N-1}f(T^{j}x)=\int_Xf(x)d\mu(x)\ ,$$ for $\mu$-almost every $x\in X$.
\end{teorema}

The conclusion of the previous theorem says that the orbit $\{T^{j}x\}_{j \ge 0}$ of $x$ is u.d. for almost every $x\in X$ whenever $T$ is ergodic.

\section{Main results}

We now turn to the main purpose of this paper showing that the Kakutani-Fibonacci sequence can be obtained as the orbit of the origin generated by an interval exchange $T: [0,1[\rightarrow [0,1[$, using the following {\it cutting-stacking} technique (see \cite{Friedman}, \cite{Lambert} and \cite{grahellia}).

\begin{definizione}[{\it Cutting-stacking technique for the Kakutani-Fibonacci sequence}]\label{8}
\rm{For every fixed $n$,  the intervals of the $n$-th Kakutani-Fibonacci partition $\alpha^{n}\omega$ are represented by two columns $C_n=\{L_n,S_n \}$ constructed as follows.

Let us start with two columns $C_1=\{L_1,S_1\}$, where $L_1=[0,\alpha[$, $S_1=[\alpha,1[$. 
 If we divide $L_1$ proportionally to $\alpha$ and $\alpha^2$, we can write $L_1=\{L_1^0, L_1^1\}$, where $L_1^0=[0,\alpha^2[$ and $L_1^1=[\alpha^2,\alpha[$.
Now we stack the interval $S_1$ over $L_1^0$ (and use the common notation $L_1^0 *S_1$) as they have the same length, called {\it width} and denoted by $w(L_1^0)=w(S_1)=\alpha^2$. So we get two new columns $C_2=\{L_2,S_2\}$, 
 where $L_2=L_1^0 *S_1$ and $S_2=L_1^1$. We denote by $b(L_2)=[0, \alpha^2[$ the {\it bottom} of $L_2$ and by $h(L_2)=2=l_2$ its {\it height} (or the number of intervals of which $L_2$ is made).

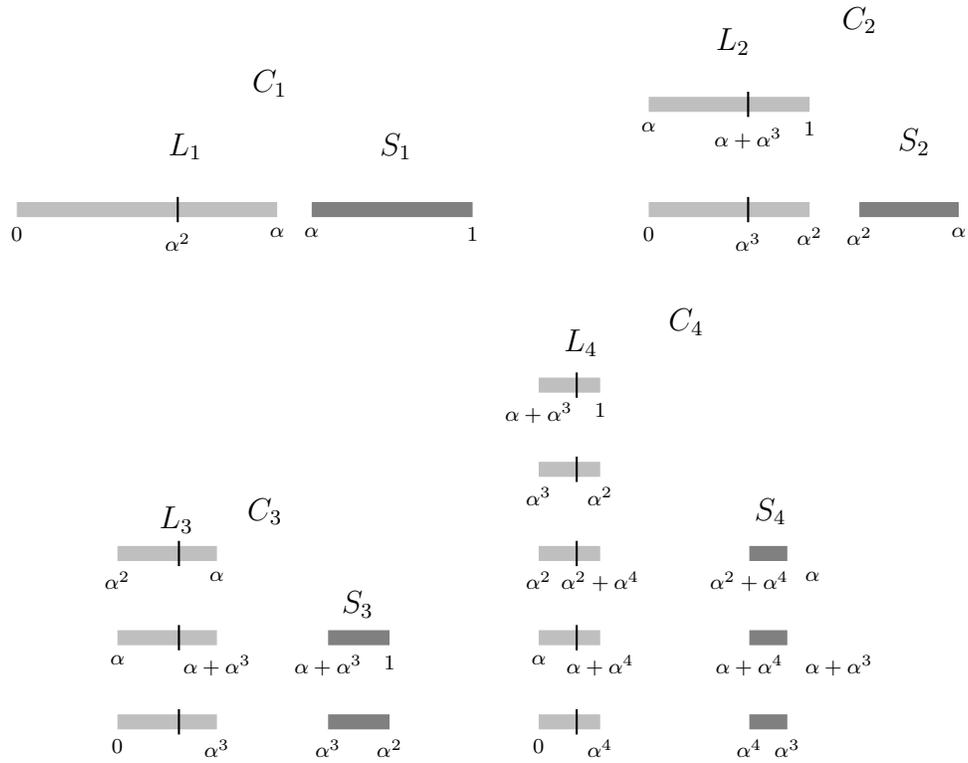
\begin{figure}[h!]
\begin{center}
\begin{tikzpicture}[scale=5.6]
\draw [lightgray, line width=0.2cm] (0,0) node[below, black]{\scriptsize $0$} -- (0.61803,0) node[below, black]{\scriptsize $\alpha$};
\node at (0.6,0.3) {$C_1$};
\node at (0.4,0.15) {$L_1$};
\node at (0.9,0.15) {$S_1$};
\draw [gray, line width=0.2cm] (0.7,0) node[below, black]{\scriptsize $\alpha$} -- (1.08196,0) node[below, black]{\scriptsize $1$};
\draw [thick] (0.38196,-0.03) node[below, black]{\scriptsize $\alpha^{2}$} -- (0.38196, 0.03);

\draw [lightgray, line width=0.2cm] (1.5,0) node[below, black]{\scriptsize $0$} -- (1.88196,0) node[below, black]{\scriptsize $\alpha^{2}$};
\draw [lightgray, line width=0.2cm] (1.5,0.25) node[below, black]{\scriptsize $\alpha$} -- (1.88196,0.25) node[below, black]{\scriptsize $1$};
\draw [gray, line width=0.2cm] (2,0) node[below, black]{\scriptsize $\alpha^{2}$} -- (2.23606,0) node[below, black]{\scriptsize $\alpha$};
\draw [thick] (1.73606,-0.03) node[below, black]{\scriptsize $\alpha^{3}$} -- (1.73606, 0.03);

\draw [thick,black] (1.73606, 0.22) node[below, black]{\scriptsize $\alpha+\alpha^{3}$} -- (1.73606, 0.28);

\node at (2,0.45) {$C_2$};
\node at (1.7,0.4) {$L_2$};
\node at (2.13,0.16) {$S_2$};

\end{tikzpicture}
\end{center}

\begin{center}
\begin{tikzpicture}[scale=5.6]

\draw [lightgray, line width=0.2cm] (0,0) node[below, black]{\scriptsize $0$} -- (0.23606,0) node[below, black]{\scriptsize $\alpha^{3}$};
\draw [thick,black] (0.14589, -0.03) -- (0.14589, 0.03);
\draw [lightgray, line width=0.2cm] (0,0.2) node[below, black]{\scriptsize $\alpha$} -- (0.23606,0.2) node[below, black]{\scriptsize $\alpha+\alpha^{3}$};
\draw [thick,black] (0.14589, 0.17) -- (0.14589, 0.23);
\draw [lightgray, line width=0.2cm] (0,0.4) node[below, black]{\scriptsize $\alpha^{2}$} -- (0.23606,0.4) node[below, black]{\scriptsize $\alpha$};
\draw [thick,black] (0.14589, 0.37) -- (0.14589, 0.43);

\draw [gray, line width=0.2cm] (0.5,0) node[below, black]{\scriptsize $\alpha^{3}$} -- (0.64589,0) node[below, black]{\scriptsize $\alpha^{2}$};
\draw [gray, line width=0.2cm] (0.5,0.2) node[below, black]{\scriptsize $\alpha +\alpha^{3}$} -- (0.64589,0.2) node[below, black]{\scriptsize $1$};

\node at (0.35,0.5) {$C_3$};
\node at (0.14,0.48) {$L_3$};
\node at (0.57, 0.28){$S_3$};

\draw [lightgray, line width=0.2cm] (1,0) node[below, black]{\scriptsize $0$} -- (1.14589,0) node[below, black]{\scriptsize $\alpha^{4}$};
\draw [thick,black] (1.09016, -0.03) -- (1.09016, 0.03);
\draw [lightgray, line width=0.2cm] (1,0.2) node[below, black]{\scriptsize $\alpha$} -- (1.14589,0.2) node[below, black]{\scriptsize $\alpha +\alpha^{4}$};
\draw [thick,black] (1.09016, 0.17) -- (1.09016, 0.23);
\draw [lightgray, line width=0.2cm] (1,0.4) node[below, black]{\scriptsize $\alpha^{2}$} -- (1.14589,0.4) node[below, black]{\scriptsize $\alpha^{2}+\alpha^{4}$};
\draw [thick,black] (1.09016, 0.37) -- (1.09016, 0.43);
\draw [lightgray, line width=0.2cm] (1,0.6) node[below, black]{\scriptsize $\alpha^{3}$} -- (1.14589,0.6) node[below, black]{\scriptsize $\alpha^{2}$};
\draw [thick,black] (1.09016, 0.57) -- (1.09016, 0.63);
\draw [lightgray, line width=0.2cm] (1,0.8) node[below, black]{\scriptsize $\alpha+\alpha^{3}$} -- (1.14589,0.8) node[below, black]{\scriptsize $1$};
\draw [thick,black] (1.09016, 0.77) -- (1.09016, 0.83);

\draw [gray, line width=0.2cm] (1.5,0) node[below, black]{\scriptsize $\alpha^{4}$} -- (1.59016,0) node[below, black]{\scriptsize $\alpha^{3}$};
\draw [gray, line width=0.2cm] (1.5,0.2) node[below, black]{\scriptsize $\alpha +\alpha^{4}$} -- (1.59016,0.2) node[below right, black]{\scriptsize $\alpha +\alpha^{3}$};
\draw [gray, line width=0.2cm] (1.5,0.4) node[below, black]{\scriptsize $\alpha^{2} +\alpha^{4}$} -- (1.59016,0.4) node[below right, black]{\scriptsize $\alpha$};

\node at (1.35,0.95) {$C_4$};
\node at (1.1,0.9) {$L_4$};
\node at (1.55, 0.5){$S_4$};

\end{tikzpicture}

\caption{Partial graph of the cutting-stacking method for $\{\alpha^n\omega\}_{n\ge 1}$}
\end{center}
\end{figure}

 If we continue this way, at the $n$-th step we get two columns denoted by $C_n=\{L_n,S_n\},$
with $h(L_n)=l_n$ and $h(S_n)=s_n$.  We divide $L_n$ into two columns, say $L_n^0$ and $L_n^1$, where $w(L_n^0)=\alpha^{n+1}$ and $w(L_n^1)=\alpha^{n+2}$, and stack $S_n$ over $L_n^0$, obtaining 
$$C_{n+1}=\{L_{n+1},S_{n+1}\},$$ where
\begin{equation*}L_{n+1}=L_n^0*S_n \qquad {\rm and}\qquad S_{n+1}=L_n^1\ ,\end{equation*} with
$$w(L_{n+1})=\alpha^{n+1}\ ,\quad h(L_{n+1})=l_{n+1}\ , \quad b(L_{n+1})=[0,\alpha^{n+1}[$$ and
$$w(S_{n+1})=\alpha^{n+2}\ , \quad h(S_{n+1})=s_{n+1}\ ,\quad b(C_{n+1})=[\alpha^{n+1},\alpha^n[\ .$$
}
\end{definizione}

\bigskip

The first steps of the procedure are visualized in Figure 2.\\

To the above cutting-stacking construction it is naturally associated an interval exchange, whose explicit expression is given by the following result.

\begin{proposizione}\label{9}
The interval exchange corresponding to the cutting-stacking procedure described in Definition \ref{8} is the map $T:[0,1[\rightarrow [0,1[$ whose restriction to $ I_{k}$ is  $T_k$, where
 $$T_1(x)=x+\alpha \quad {\rm if\ }\ x \in I_1=[0,\alpha^{2}[ \,$$
and, for every $k\geq 1$,
$$T_{2k}(x)=x+\alpha^{2k}-\sum_{j=0}^{k-1}\alpha^{2j+1} \quad {\rm if}\ x \in I_{2k}=\left[\sum_{j=0}^{k-1}\alpha^{2j+1}, \sum_{j=0}^{k}\alpha^{2j+1}\right[$$
and
$$T_{2k+1}(x)=x+\alpha^{2k+1}-\sum_{j=0}^{k-1}\alpha^{2(j+1)} \quad {\rm if}\ x \in I_{2k+1}= \left[\sum_{j=0}^{k-1}\alpha^{2(j+1)}, \sum_{j=0}^{k}\alpha^{2(j+1)}\right[.$$
\end{proposizione}

\begin{proof} 
If we write $T(x)=x+c_k$ whenever  $x \in I_k$ for any $k \ge 1$, we simply observe that $I_k+c_k=[\alpha^k, \alpha^k+\alpha^{k+1} [ $ for all $k \ge 1$.  Therefore, $\lambda(\cup_{k\geq 0}(I_k+c_k))=1$ and $(I_h+c_h)\cap (I_k+c_k) =\emptyset $ whenever $h \neq k$, which proves that $T$ is an interval exchange.

Now we show how the map $T$ acts in the cutting-stacking procedure, proving that in each column $C_n=\{L_n, S_n \}$ the transformation $T$ maps each interval of $L_n$ (respectively, $S_n$) onto the interval above it and the {\it top} interval of $L_n^0$, denoted as usual by $top(L_n^0)$, onto the bottom interval of $S_n$. 

 Let us start with $C_1=\{L_1, S_1 \}$. When we divide the column $L_1$ into two sub-columns $L^0_1$ and $L_1^1$, made by one interval each, we notice that $top(L^0_1)=I_1$ and $b(S_1)=T(I_1)$. Therefore, we stack $S_1$ onto $L_1^0$ using $T_1$, which maps $I_1$ onto $b(S_1)$ and gives birth to the columns $C_2=\{ L_2, S_2\}$, where $L_2=L_1^0*S_1$ and $S_2=L_1^1$.

Now we consider $n=2k$ and $C_{2k}=\{ L_{2k}, S_{2k}\}$. When we divide $L_{2k}$ proportionally to $\alpha$ and $\alpha^2$, obtaining therefore $L_{2k}^0$ and $L_{2k}^1$, we notice that $b(S_{2k})=T_{2k}(I_{2k})=[\alpha^{2k}, \alpha^{2k+1}[$ and, consequently, $top(L_{2k}^0)=I_{2k}$ because $T_{2k}$ is a bijection. In other words, $T_{2k}$ stacks the bottom of $S_{2k}$ onto the top of $L_{2k}^0$ because $T_{2k}(top(L_{2k}^0))=b(S_{2k})=[\alpha^{2k}, \alpha^{2k-1}[$\ .

A simple calculation shows that if we consider the case $n={2k+1}$ we have  $T_{2k+1}(I_{2k+1})=b(S_{2k+1})=[\alpha^{2k+1}, \alpha^{2k}[$ and, therefore,  $top(L_{2k+1}^0)=I_{2k+1}$. 

The proposition is now completely proved. \end{proof}

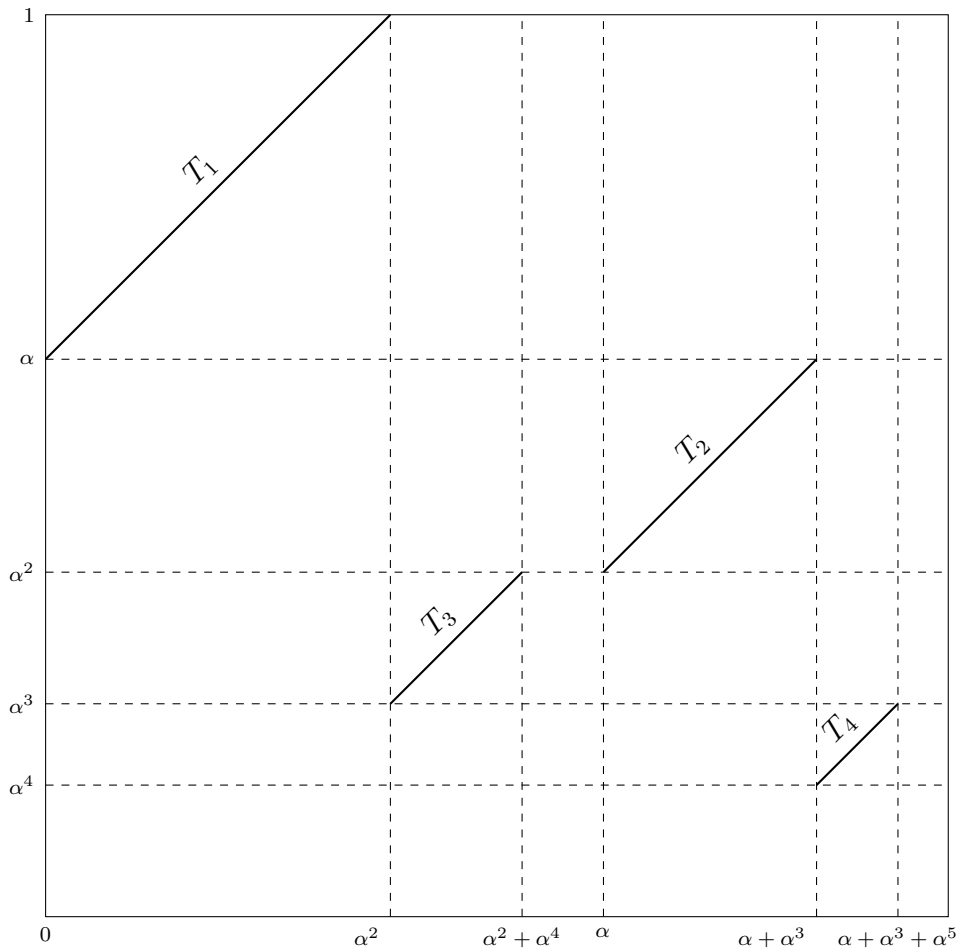
\begin{figure}[h!]
\begin{center}
\begin{tikzpicture}[scale=12]
\draw (0,0) node[below]{\scriptsize 0} -- (1,0) -- (1,1) -- (0,1)node[left]{\scriptsize1} -- (0,0);
\draw [thick] (0,0.61803) node[left]{\scriptsize $\alpha$} -- (0.38196,1)
node[midway, sloped,above] {$T_1$};;
\draw [thick] (0.38196,0.23606) -- (0.52786,0.38196)
node[midway, sloped,above] {$T_3$};;

\draw [thick] (0.61803,0.38198) -- (0.85410,0.61803)
node[midway, sloped,above] {$T_2$};
\draw [thick] (0.85410,0.14589) -- (0.94427,0.23606)
node[midway, sloped,above] {$T_4$};;
\draw [dashed] (0.61803,0) node[below]{\scriptsize $\alpha^{}$} -- (0.61803,1); 
\draw [dashed] (0.38196,0) node[below left]{\scriptsize $\alpha^{2}$} -- (0.38196,1);
\draw [dashed] (0.52786,0) node[below]{\scriptsize $\alpha^{2}+\alpha^{4}$} -- (0.52786,1);
\draw [dashed] (0.85410,0) node[below left]{\scriptsize $\alpha+\alpha^{3}$} -- (0.85410,1);
\draw [dashed] (0.94427,0) node[right, below]{\scriptsize $\alpha+\alpha^{3}+\alpha^{5}$} -- (0.94427,1);
\draw [dashed] (0,0.61803) -- (1,0.61803);
\draw [dashed] (0,0.38196)  node[left]{\scriptsize $\alpha^{2}$} -- (1,0.38196);
\draw [dashed] (0,0.23606) node[left]{\scriptsize $\alpha^{3}$}  -- (1,0.23606);
\draw [dashed] (0,0.14589) node[left]{\scriptsize $\alpha^{4}$}  -- (1,0.14589);
\end{tikzpicture}
\end{center}
\caption{Partial graph of the Kakutani-Fibonacci transformation $T$}
\end{figure}

\begin{definizione}\label{10}
{\rm The map $T:[0,1[\rightarrow [0,1[$, whose expression is given in Proposition \ref{9}, is called the \emph{Kakutani-Fibonacci transformation}. }
\end{definizione}
Figure 3 shows the graph of the maps $T_k$, with $1 \leq k\leq 4$.

Note that in \cite{grahellia} it is presented a transformation $T_{f}$, called  the Fibonacci transformation, related to the translation $x \rightarrow x+\alpha (mod\ 1)$ by means of the cutting-stacking technique. At each step of the cutting-stacking procedure, the two columns obtained by means of $T_{f}$ have the same height and  width as the two columns obtained by means of the Kakutani-Fibonacci transformation $T$, but the bottoms are different as well as the order of the intervals in each column. Moreover, $T_{f}$ is  not defined in $0$.

\begin{teorema}\label{ergodicity}
The Kakutani-Fibonacci transformation $T$ is ergodic.
\end{teorema}
\begin{proof} The arguments we use are inspired by  \cite{Friedman}.

Let us denote by $C_m^*$ the {\it stack} made by  $L_m*S_m$ and let us define the mapping $\tau_m: [0,1[\rightarrow [0,1[$ as follows.
For $1 \le i \le l_m-1$ it is the translation of the interval $J_i$ of $L_m$ onto $J_{i+1}$, the interval just above $J_i$, and hence it coincides with $T$.
If $i=l_m$,  it is a contraction by parameter $\alpha$ of $top(L_m)$ onto $b(S_m)$.
For $1 \le i \le s_m-1$ it is the translation of the interval $\tilde J_i$ of $S_m$ onto $\tilde{J}_{i+1}$, the interval just above $\tilde J_i$, and hence it coincides with $T$ again.

Now we fix a measurable set $B$ in $[0,1[$ with $\lambda(B)>0$, such that $T^{-1}(B)=B$, and we prove that  $\lambda(B^T)=1$, where $B^T=\bigcup_{i=- \infty}^{+ \infty} T^i(B)$.

From Lebesgue density theorem, there exists a point $x_0$ having density $1$ for $B$. 
This implies that for every fixed $\epsilon>0$ there exists $\delta>0$ such that if $I$ is any interval containing $x_0$ with $\lambda(I)< \delta$, we have

\begin{equation}\label{2}\lambda(B \cap I)>(1-\epsilon) \lambda(I).
\end{equation}

\smallskip
Due to the fact that the diameter of the Kakutani-Fibonacci sequence of partition $(\alpha^m \omega)_{m \ge 1}$ tends to $0$ when $m \rightarrow \infty$, there exists $m >0$ such that $\alpha^m< \delta$  and, therefore, there exists an interval of $C^*_m$ for which (\ref{2}) holds. For sake of brevity, we will denote it by $I$.
 
As $C^*_m= L_m* S_m$, the interval $I$ could belong to the column $L_m$ or to the column $S_m$. If $I \in S_m$, it is equivalent to say that $I \in {L_{m+1}}$. For this reason, without any loss of generality we can suppose that $I \in L_m$.

Suppose now that $I$ is the $r$-th interval from below in the column $L_m$ and observe that 
\begin{equation*} \lambda \left (\bigcup_{i=-r+1}^{t_m-r} \tau_m^i(I) \right)=\sum_{i=-r+1}^{t_m-r}\lambda  (I) =1.  
\end{equation*}

\smallskip
Taking the above identity and (\ref{2}) into account, we have the following inequalities:
\begin{eqnarray*}\lambda(B^T)&\ge&\lambda \left((B \cap I)^T\right)  \\
&\ge&\lambda \left(  \bigcup_{i=1-r}^{l_m-r} T^i(B \cap I) \right)+\lambda \left(  \bigcup_{j=0}^{s_m-1} \left( T^j\left (B \cap b(S_m)\right)\right) \right)\\
&=&\lambda \left(  \bigcup_{i=1-r}^{t_m-r}  \tau_m^i(B \cap I) \right)+\lambda \left(  \bigcup_{j=l_m+1-r}^{t_m-r}  \tau_m^j(B \cap I) \right)\\
&=& \sum_{i=1-r}^{l_m-r} \lambda \left( \tau_m^i(B \cap I)\right)+\sum_{j=l_m+1-r}^{t_m-r} \lambda \left( \tau_m^j(B \cap I)\right)\\
&>&(1-\epsilon)\ l_m\ \alpha^{m}+(1-\epsilon)\ s_m \ \alpha^{m+1}=1-\epsilon.
\end{eqnarray*}
As $\epsilon$ is arbitrary, we conclude that $\lambda(B)=1$. Therefore, $T$ is ergodic and the theorem is proved.
\end{proof}

A direct consequence of the above theorem and Theorem \ref{B} is the following

\begin{teorema}\label{b-k-f} The sequence $\{T^{n}(x)\}_{n\geq 0}$ is u.d. for almost every $x \in [0,1[$.
\end{teorema}

We conclude this paper showing how to get the points of the Kakutani-Fibonacci sequence by means of the  Kakutani-Fibonacci transformation.

\begin{teorema}\label{15}
The Kakutani-Fibonacci sequence of points  $\{\xi_n\}_{n \ge 1}$ coincides with $\{T^{n}(0)\}_{n\geq 0}$.
\end{teorema}
\begin{proof}  We want to show  by induction on $n\ge 1$ that  

\begin{equation}\label{3} \big( \xi_1, \ \xi_2, \dots,\ \xi_{t_n})=\big(0, T(0), T^2(0), \dots, T^{t_{n}-1}(0)\big).
\end{equation}

\bigskip
 If $n=1$,  (\ref{3}) is obviously verified.

We suppose that (\ref{3}) is true for  $n$ and prove that
\begin{eqnarray*} &&\big(\ \xi_1, \ \xi_2, \dots,\ \xi_{t_n}, \phi_{n+1}(\xi_1),\dots ,\phi_{n+1}(\xi_{l_n}))\\
=&&\big(0, T(0), T^2(0), \dots, T^{t_{n}-1}(0), T^{t_n}(0), T^{t_n+1}(0), \dots, T^{t_{n}+l_n-1}(0)\big).
\end{eqnarray*}

Due to (\ref{1}) and the inductive assumption, it is sufficient to prove that
 \begin{eqnarray*}&&\big(\phi_{n+1}(0),\phi_{n+1}(T(0)), \dots ,\phi_{n+1}(T^{t_n-1}(0))\big)\\
= &&\big(T^{t_n}(0), T^{t_n+1}(0), \dots, T^{t_{n}+l_n-1}(0)\big).\end{eqnarray*}

As $\phi_{n+1}(x)=x+\alpha^{n+1}$, the above identity  is equivalent to  
 \begin{equation*}\big(\alpha^{(n+1)},T(0)+\alpha^{(n+1)}, \dots ,T^{t_n-1}(0)+\alpha^{(n+1)}\big)\end{equation*}
\begin{equation}\label{4}=\big(T^{t_n}(0), T^{t_n+1}(0), \dots, T^{t_{n}+l_n-1}(0)\big).\end{equation}

\bigskip
In order to prove (\ref{4}), we focus our attention on the intervals of the column $S_{n+1}$ and on their left endpoints.

We note that, due to the cutting-stacking procedure described in Definition \ref{8} and to the nature of $T$ described in the proof of Proposition \ref{9}, and specifically to the fact that  $b(S_{n+1})=T(top(L_{n+1}))=T(T^{l_{n}+s_n-1}(\ [0, \alpha^{n+1}[\ ) )$,  
the columns $C_{n+1}=\{L_{n+1}, S_{n+1} \}$ can be written as follows:
\begin{eqnarray*} L_{n+1}&&=L_n^0*S_n=\Big(b(L_n^0)* \cdots *top(L_n^0)\Big)*\Big(b(S_n)*\cdots*top(S_n)\Big)\\
&&=\Big([0, \alpha^{n+1}[ \ * \cdots *\ T^{l_{n}-1}(\ [0, \alpha^{n+1}[\ )\ \Big)*\\
  &&\ * \ \Big( T^{l_{n}}(\ [0, \alpha^{n+1}[\ )\ *\cdots\ *\ T^{l_{n}+s_n-1}(\ [0, \alpha^{n+1}[\ ) \Big)
 \end{eqnarray*}
and
 \begin{eqnarray*} S_{n+1}=L_n^1&&=b(L_n^1)* \cdots *top(L_n^1)\\
 &&=\Big([\alpha^{n+1}, \alpha^{n}[\ * \cdots *\ T^{l_{n}-1}\left (\  [\alpha^{n+1}, \alpha^{n}[\ \right  ) \Big)\\
&&= \Big( T^{t_{n}}(\ [0, \alpha^{n+1}[\ )\ *\cdots\ *\ T^{t_{n}+l_n-1} \left(\ [0, \alpha^{n+1}[\ \right) \Big).
 \end{eqnarray*}

 \bigskip
 Therefore, as $T$ is an interval exchange, the $t_n$ left endpoints of  $L_{n+1}$ are $0, T(0), \cdots, T^{t_n-1}(0) $ and the left endpoints of $S_{n+1}$, which are a right shift by the constant $\alpha^{n+1}$ of the left endpoints of $L_{n+1}$, are $T^{t_n}(0), T^{t_n+1}(0),$ $ \cdots, T^{t_n+s_n-1}(0)$, which proves (\ref{4}).
 
The theorem is now completely proved. \end{proof}

\noindent {\bf Acknowledgements}
The authors wish to  express their thanks to Robert Tichy for the stimulating discussion which took place in Graz.

\bigskip
\noindent {\it Ingrid Carbone, Maria Rita Iac\`{o}*, Aljo\v{s}a Vol\v{c}i\v{c} }\\ 
University of Calabria, Department of Mathematics\\ 
Ponte P. Bucci Cubo 30B\\
87036 Arcavacata di Rende (Cosenza), Italy\\
E-mail: \texttt{i.carbone@unical.it}} \\
E-mail: \texttt{volcic@unical.it}\\
E-mail: \texttt{iaco@mat.unical.it}

\bigskip
\noindent * 
  Research supported by the Doctoral Fellowship in Mathematics and Informatics of University of Calabria in cotutelle with Graz University of Technology, Institute of Mathematics A, Steyrergasse 30, 8010 Graz, Austria. 

\end{document}